\documentclass[11pt]{article}   % If the bibliography is changed:
\newtheorem{theorem}{Theorem} [section]
\newtheorem{corollary}[theorem]{Corollary}

\newtheorem{lemma}[theorem]{Lemma}

\newenvironment {proof} {{\it Proof.}}{\hspace*{\fill}$\Box$\par\vspace{4mm}}
\usepackage{amssymb}
\usepackage{amsmath}
\usepackage{amsfonts}
\usepackage{graphicx,psfrag}
\usepackage{tikz}
\begin{document}

\bibliographystyle{plain}

\title{Adjacent $q$-Cycles in Permutations}

 \author{Richard A. Brualdi\\
 Department of Mathematics\\
 University of Wisconsin\\
 Madison, WI 53706\\
 {\tt brualdi@math.wisc.edu}
 \and
 Emeric Deutsch\\
 Department of Mathematics\\
 Polytechnic Institute of NYU\\
 Brooklyn, NY 11201\\
 {\tt emericdeutsch@msn.com}
  }

\maketitle

 \begin{abstract}
We introduce  a new permutation statistic, namely, the number of cycles of length $q$ consisting of consecutive integers, and  consider the distribution of this statistic among the permutations of $\{1,2,\ldots,n\}$.  We determine explicit formulas, recurrence relations, and ordinary and exponential generating functions. A generalization to more than one fixed length is also considered.

\medskip
\noindent {\bf Key words:  permutations, derangements, adjacent cycles.}
 
 \smallskip
\noindent {\bf AMS subject classifications: 05A05, 05A10, 05A15 . } 
\end{abstract}

\section{Introduction}
Let $S_n$ denote the set of all permutations of $[n]=\{1,2,\ldots,n\}$.
We use the one-line notation for permutations but then write a permutation according to its standard (disjoint) cycle decomposition. For example,
\[\pi =432157869=(14)(23)(5)(678)(9)\]
is a permutation in $S_9$ decomposed into cycles of lengths 2,2,1,3,1. (We omit commas between the integers in a cycle unless it leads to ambiguity.) We standardize the cycle decomposition by writing, for each cycle, its smallest element first and then writing the cycles in order of their smallest element. Generalizing the notion of an adjacent transposition (a 2-cycle), we define an {\it adjacent $q$-cycle}, abbreviated A$q$C, to be a $q$-cycle in which the elements in the cycle form a consecutive set of integers:
\[ (a, a+1,a+2,\ldots,a+q-1).\]
The permutation $\pi$ above has two A1C's (the fixed points (5) and (9)), one A2C (the adjacent transposition
$(23)$), and one A3C ($(678)$).
 A permutation in $S_n$ is {\it adjacent $q$-cycle free}, abbreviated
{\it A$q$C-free}, provided that it does not have any AqC's.  For example,  $532146=(154)(23)(6)$ is A3C-free. An A1C-free permutation is just a usual derangement.

We introduce a new permutation statistic for $S_n$ by considering, for an integer $q$ with $1\le q\le n$, the {\it number of A$q$C's} contained in a permutation, and consider the distribution of this statistic on $S_n$.  More precisely, let
\[a(n,k)=|\{\pi\in S_n: \pi \mbox{ has exactly $k$ AqC's}\}|.\]
Let $r=\left\lfloor\frac{n}{q}\right\rfloor$. If $k>r$, then  $a(n,k)= 0$.  We are interested in the sequence of numbers
\begin{equation}\label{eq:def}
a(n,0),a(n,1), a(n,2),\ldots,a(n,r),\end{equation}
where note that, for simplicity,  the parameter $q$ is not reflected in the notation. For example,
if $n=3$ and $q=2$, then $r=1$ and we have
$a(3,0)=4$ and  $a(3,1)=2$, with these numbers coming from, respectively,  the permutations
\[(1)(2)(3), (13)(2), (123),(132)\mbox{ and }  (1)(23), (12)3.\]

If $q=1$, and so $r=n$, the sequence
(\ref{eq:def}) gives the number of permutations with a specified number of fixed points, and these numbers are the {\it rencontres numbers} (see p.~59 in \cite{riordan}):
\[a(n,0),a(n,1),a(n,2),\ldots,a(n,n) \quad (q=1).\]
The sequence of {\it derangement numbers} $d_0,d_1,d_2,\ldots$ (see e.g. Sec. 6.3 in \cite{rab}) is the sequence 
\[a(0,0),a(1,0),a(2,0), \ldots \quad (q=1)\]
and, as a simple combinatorial argument shows,
\[a(n,k)={n\choose k}d_{n-k} \quad (q=1).\]

We now summarize the contents of this paper.
In Section 2,  we derive  a formula for the numbers $a(n,k)$ and  use this formula to obtain  $a(n,k)$ from the numbers $a(n,k-1)$; we also discuss  the number of A$q$C-free permutations.
In Section 3, we derive a homogeneous recurrence relation for $a(n,k)$. In Section 4,  we determine the ordinary and exponential generating functions for the sequence (\ref{eq:def}). In Section 5, we obtain a formula for the number of permutations having a specified number of  adjacent cycles of each of several lengths. Finally, in Section 6,  we 
illustrate how the permanent function can be used to calculate some of our numbers. 

\section{A formula}

Again we fix $q$ throughout this section and suppress it in our notation.

\begin{theorem}\label{th:formula}
We have
\begin{equation}\label{eq:formula}
a(n,k)=\sum_{j=k}^{\left\lfloor\frac{n}{q}\right\rfloor}(-1)^{k+j}{j\choose k}
\frac{(n-(q-1)j)!}{j!}, \quad \left(k=0,1,\ldots, \left\lfloor\frac{n}{q}\right\rfloor\right).
\end{equation}
\end{theorem}

\begin{proof}
We again set $r=\lfloor n/q\rfloor$, and we first determine a system of $r+1$ linear equations that the numbers $a(n,0),a(n,1),\ldots,a(n,r)$ satisfy.
This system is the triangular system:
\begin{eqnarray}\label{eq:newtriangle} \nonumber
{0\choose 0}a(n,0)+{1\choose 0}a(n,1)+{2\choose 0}a(n,2)+\cdots+
{r\choose 0}a(n,r)&=&\frac{n|}{0!}\\ \nonumber
{1\choose 1}a(n,1)+{2\choose 1}a(n,2)+\cdots+{r\choose 1}a(n,r)&=&
\frac{(n-(q-1))!}{1!}\\ \nonumber
{2\choose 2}a(n,2)+\cdots + {r\choose 2}a(n,r)&=&\frac{(n-2(q-1))!}{2!}
\\
\cdots &=& \cdots\\ \nonumber
{r\choose r}a(n,r)&=&\frac{(n-r(q-1))!}{r!}
\end{eqnarray}
To verify that the equations (\ref{eq:triangle}) hold, let $t$ be an integer with $0\le t\le r$ and let
${\cal M}_n^t$ be the multiset of permutations $\pi$ in $S_n$ obtained
by  (i) selecting $t$ A$q$C's, thereby filling in $tq$ of the positions in the 
one-line representation of $\pi$, and then (ii) filling in the remaining $n-tq$ 
positions in all possible ways. This indeed defines a multiset and not a set since, for instance, if $n=6$, $q=2$, and $t=2$, the permutation $214365$ occurs three times: 
choose in (i) any pair of the 2-cycles (12), (34), and (56) 
(say (12) and (56) giving $21- - 65$), and then
fill in the remaining two positions to get $214365$.

We now determine the cardinality of the multiset ${\cal M}_n^t$ by counting in two different ways, and then equate the two counts to yield (\ref{eq:newtriangle}). Suppose that $t$ A$Q$C's have been placed in the one-line representation of a permutation. Let $y_1,y_2,\ldots,y_{t+1}$ be the number of contiguous empty positions before, in-between, and after the occupied positions. For example, taking  $n=13$, $q=2$, and $t=3$,  if the selected $t=3$ A$2$C's are $(45)$,  $(67)$, and $(ab)$,  where $a=12$ and $b=13$, then we have
\[- - - 5 4 7 6 - - - - b a,\]
and therefore, $y_1=3, y_2=0,y_3=4$ and $y_4=0$.
In general, we have
\[y_1+y_2+\cdots+y_{t+1}=n-tq\]
and this equation has
\[{{n+t-tq}\choose t}\]
nonnegative integral solutions (see e.g. Sec 2.5 in  \cite{rab}).
Thus
\[|{\cal M}_n^t|={{n+t-tq}\choose t} (n-tq)!= \frac{(n-t(q-1))!}{t!}
\quad (t=0, 1, \ldots, r).\]

For $i=t,t+1,\ldots,r$, let ${\cal M}_n^t{(i)}$ be the multisubset of ${\cal M}_n^t$ consisting of all those permutations $\pi$ with exactly $i$ A$q$C's. Each permutation in ${\cal M}_n^t{(i)}$ has $t$ initially placed A$q$C's and $i-t$ additional A$Q$C's obtained by filling in the remaining positions. Such a permutation occurs $i\choose t$ times, since any $t$ of the $i$ A$q$C's can be the initially placed A$q$C's.
Thus
\[|{\cal M}_n^t{(i)}|={i\choose t}a(n,i).\]
Since
\[|{\cal M}_n^t|=\sum_{i=t}^r|{\cal M}_n^t{(i)}|,\]
we have
\[\sum_{i=t}^r {i\choose t}a(n,i)=\frac{(n-(q-1)t)!}{t!}\quad (t=0,1,\ldots, r),\]
the equations of the triangular system (\ref{eq:triangle}). 
The inverse of the coefficient matrix  $A=\left[{j\choose i}: 0\le i,j\le r\right]$ 
is $A^{-1}=\left[(-1)^{j+i}{j\choose  i}: 0\le i,j\le r\right]$, and this gives the
 formula (\ref{eq:formula}).
\end{proof}

Special cases of formula (\ref{eq:formula}) are 
the classical formulas for the derangement and rencontres
 numbers.

\begin{corollary} Let $q=1$.
The rencontres numbers $a(n,k)$ satisfy
\[a(n,k)={n\choose k}d_{n-k}\quad (k=0,1,\ldots,n).\]
\end{corollary}

\begin{proof} Taking $q=1$  in (\ref{eq:formula}), 
%we get
%\[d_n=a(n,0)=n!\sum_{j=0}^n (-1)^{j}\frac{1}{(j)!}.\]
\[a(n,k)=n!\sum_{j=k}^n (-1)^{k+j}\frac{1}{k!(j-k)!}\quad =
\frac{n!}{k!}\sum_{i=0}^{n-k} \frac{(-1)^i}{i!}.\]
If $k=0$, we get
\[d_n=a(n,0)=n!\sum_{i=0}^{n} \frac{(-1)^i}{i!}.\]
%n!\sum_{j=0}^n(-1)^j\frac{1}{j!}.
Thus, for the rencontres numbers we get again
\[a(n,k)={n\choose k}d_{n-k}.\]
\end{proof}

For each value of $q$, the numbers $a(n,k)$ determine a Pascal-like ``triangle.'' For instance, with $q=5$, $a(n,k)$ counts the number of permutations in $S_n$ with $k$ A5C's giving the following data:
\[\begin{array}{c||rrr}
n&k=\qquad 0&1&2\\ \hline\hline
0&1&0&0\\
1&1&0&0\\
2&2&0&0\\
3&6&0&0\\
4&24&0&0\\
5&119&1&0\\
6&718&2&0\\
7&5034&6&0\\
8&40296&24&0\\
9&362760&120&0\\
10&3628081&718&1\\
11&39911763&5034&3\\
12&478961292&40296&12\\
13&6226657980&362760&60\end{array}\]

In the following theorem, we establish a simple relationship 
between the entries of two consecutive columns of these triangles.

\begin{theorem}\label{th:triangle}
We have
\begin{equation}\label{eq:triangle}
a(n+q-1,k)=\left\{\begin{array}{ll}
\frac{1}{k} \left(a(n,k-1) +(-1)^{k+\frac{n}{q}}{{\frac{n}{q}}\choose {k-1}}\right),&\mbox{ if } q|n\\
\frac{1}{k}a(n,k-1), & \mbox{ otherwise.}
\end{array}\right.\end{equation}
\end{theorem}

\begin{proof}
We use the formula (\ref{eq:formula}) for $a(n,k)$ in Theorem \ref{th:formula} to get, after elementary manipulation and change of summation variable,
\begin{equation}\label{eq:triangle2} a(n+q-1,k)-\frac{1}{k}a(n,k-1)=
\sum_{i=k-1}^{\left\lfloor\frac{n-1}{q}\right\rfloor}\lambda(q,n,k,i)-
\sum_{i=k-1}^{\left\lfloor\frac{n}{q}\right\rfloor}\lambda(q,n,k,i),\end{equation}
where
\[\lambda(q,n,k,i)=\frac{(-1)^{i+k+1}\left(n-(q-1)i\right)!}{i!k}{i\choose{k-1}}.\]
If $q$ does not divide $n$, then the upper summation limits in (\ref{eq:triangle2}) are equal, implying that $a(n+q-1)=\frac{1}{k}a(n,k-1)$.
If $q$ divides $n$, then $\lfloor\frac{n}{q}\rfloor=\lfloor\frac{n-1}{q}\rfloor+1$, and we obtain
\[a(n+q-1,k)-\frac{1}{k}a(n,k-1)=\lambda(q,n,k,\left\lfloor\frac{n}{q}\right\rfloor)=
(-1)^{k+\frac{n}{q}}{{\frac{n}{q}}\choose {k-1}},\]
completing the proof.
\end{proof}

Let $b_n=a(n,0)$ be the {\it number of A$q$C-free permutations in $S_n$}.
If $q=1$, then $b_n=d_n$, the $n$th derangement number.  Thus the  numbers $b_n$ include the derangement numbers as a special case.
The formula (\ref{eq:formula}) gives the following formula for $b_n$.

\begin{corollary}\label{cor:OK}
\begin{equation}\label{eq:bformula}
b_n=\sum_{j=0}^{\left\lfloor\frac{n}{q}\right\rfloor}(-1)^{j}
\frac{(n-(q-1)j)!}{j!}.\end{equation}
\end{corollary}

If $q=1$, then we get the following classical formula for the derangement numbers. 
\[b_n=n!\sum_{j=0}^n\frac{(-1)^j}{j!}=d_n.\]

Setting $k=1$ and  replacing $n$ by $n-q$ in (\ref{eq:triangle}),  we get the following relation for the  entries of column $k=1$  of our triangle in terms of the number of  A$q$C-free permutations.

\begin{corollary}\label{cor:triangle}
\begin{equation}\label{eq:triangle3}
a(n-1,1)=
\left\{\begin{array}{ll}
b_{n-q} +(-1)^{\frac{n}{q}}&\mbox{ if }q|n\\
b_{n-q}&\mbox{ otherwise.}
\end{array}\right.\end{equation}
\end{corollary}

%For a fixed $q$, let   $b_n=a(n,0)$, the {\it number of A$q$C-free permutations} in $S_n$. 
 It is well known (see e.g. \cite{bona}, p. 107) that $\lim_{n\rightarrow\infty}\frac{d_n}{n!}=\frac{1}{e}$.
We determine the corresponding limit for $q\ge 2$ in the next theorem.

\begin{theorem}\label{th:limit}
Let $q\ge 2$ be given. Then $\lim_{n\rightarrow\infty} \frac{b_n}{n!}=1$.
\end{theorem}

\begin{proof}
The total number of A$q$C's that occur in the permutations in $S_n$ equals $(n+1-q)!$. This is because there are  $n+1-q$ adjacent $q$-cycles 
that can be formed from $1,2,\ldots,n$ and each such cycle can be extended to $(n-q)!$ permutations in $S_n$. Let $c_n$ be the number of permutations in $S_n$ with at least one A$q$C. Then
$c_n\le (n+1-q)!$, and so
\[n!-(n+1-q)!\le n!-c_n=b_n\le n!.\]
Hence
\[1-\frac{(n+1-q)!}{n!}\le \frac{b_n}{n!}\le 1.\]
Since $q\ge 2$, the result follows.
\end{proof}

\section{Recurrence Relations}

In this section we determine a recurrence relation for the numbers $a(n,k)$. We first obtain a recurrence relation for the number $b_n=a(n,0)$
of A$q$C-free permutations of $\{1,2,\ldots,n\}.$

There is a well known recurrence relation satisfied by the derangement numbers $d_n$, namely, $d_n-nd_{n-1}=(-1)^n$. We now show that this is a special case of a recurrence relation satisfied
by $b_n$ in general.

\begin{theorem}
\label{th:recurrence}
Let $q\ge 1$ be given. Then
\begin{equation}\label{eq:recurrence}
b_n-nb_{n-1}=\left\{\begin{array}{ll}
(q-1)b_{n-q}+q(-1)^{\frac{n}{q}},&\mbox{ if $q | n$}\\
(q-1)b_{n-q},&\mbox{ otherwise}\end{array}\right. \quad (n\ge 1).\end{equation}
\end{theorem}

\begin{proof}
The proof follows by substitution of the formula for the numbers $b_n$  given in (\ref{eq:bformula}).
We verify that the recurrence relation (\ref{eq:recurrence}) holds
when $q|n$ and leave the verification to the reader if $q\not | n$.
Assume that $q|n$, and let $r=n/q$. Then by (\ref{eq:bformula}), we have
\[b_n=\sum_{j=0}^r (-1)^{j}
\frac{(n-(q-1)j)!}{j!}\mbox{ and }
b_{n-1}=\sum_{j=0}^{r-1}(-1)^{j}
\frac{(n-1-(q-1)j)!}{j!}.\]
Hence 
\begin{eqnarray}\label{eq:rec}
\nonumber
b_n-nb_{n-1}&=&(-1)^r\frac{(n-(q-1)r)!}{r!}+
\sum_{j=0}^{r-1}(-1)^{j}
\frac{(n-1-(q-1)j)! (n-(q-1)j-n)}{j!}\\
\nonumber
&=& (-1)^r-
(q-1)\sum_{j=1}^{r-1}(-1)^{j}
\frac{(n-1-(q-1)j)!}{(j-1)!}\\ 
&=& (-1)^r+
(q-1)\sum_{i=0}^{r-2}(-1)^{i}
\frac{(n-q-(q-1)i)!}{i!}.
 \end{eqnarray}
From (\ref{eq:bformula}) we also get
\begin{equation}\label{eq:rec2}
b_{n-q}=\sum_{j=0}^{r-1}(-1)^j\frac{(n-q-(q-1)j)!}{j!}=
(-1)^{r-1}+\sum_{j=0}^{r-2}(-1)^j\frac{(n-q-(q-1)j)!}{j!}.\end{equation}
Combining (\ref{eq:rec}) and (\ref{eq:rec2}), we get
\[b_n-nb_{n-1}=(-1)^r+(q-1)(b_{n-q}-(-1)^{r-1})=(q-1)b_{n-q}+q(-1)^r.\]
\end{proof}

We now derive a recurrence relation for the numbers $a(n,k)$.
% of permutations of $\{1,2,\ldots,n\}$ with $k$ A$q$C's.

\begin{theorem}\label{th:recurrence2}
For $k\ge 1$,
\begin{equation}\label{eq:recurrence2}
a(n+1,k)=a(n-q+1,k-1)+(n-qk+1)a(n,k)-a(n-q+1,k)+q(k+1)a(n,k+1).
\end{equation}
\end{theorem}

\begin{proof}
We denote a permutation in $S_n$ with $k$ A$q$C's by $\pi_{n,k}$.
The permutations in $S_{n+1}$ can be gotten from the permutations in $S_n$ by inserting $n+1$ in any of the $n+1$ places before, between, and after the  $n$ entries. If one uses the cycle decomposition of a permutation,  the permutation of $S_{n+1}$ can be gotten from the permutations in $S_n$ by inserting $n+1$ in any of the places after the entries of each cycle, or by creating a new cycle of length 1 (a fixed point).  A permutation $\pi_{n+1,k}$  can be gotten 
in this way only from permutations $\pi_{n,k+1}$ (by killing one of the A$q$C's), $\pi_{n,k}$ (by not killing its A$q$C's  and not creating any new A$q$C),  and $\pi_{n,k-1}$ (by not killing its A$q$C's and creating one new A$q$C).
We now compute the number of $\pi_{n+1,k}$'s obtained in each of these three ways.

\smallskip\noindent
{\it Case $\pi_{n+1,k}$'s from $\pi_{n,k+1}$'s}: There are $k+1$ $q$-cycles, each of which can be killed (by making it into a $(q+1)$-cycle) in $q$ ways. In this way we get 
\begin{equation}\label{eq:sum1}
q(k+1)a(n,k+1)\end{equation}
 permutations with $k$ A$q$C's.

\smallskip\noindent Case  {\it $\pi_{n+1,k}$'s from $\pi_{n,k}$'s}: Here we have to 
distinguish two possibilities according to whether or not $n$ belongs to an A$(q-1)$C. Assume first that $n$ belongs to an A$(q-1)$C, so that 
 $q\ge 2$ and  the A$(q-1)$C  containing $n$ is $(n-q+2,n-q+3,\ldots,n)$.  Then $n+1$ can be inserted as a cycle of length 1, or after each entry not in the $k$ A$q$C's except after the last entry of  $(n-q+2,n-q+3,\ldots,n)$.
 Thus, for each such $\pi_{n,k}$ we get
$(1+n-qk-1)$ permutations with $k$ A$q$C's, giving a total of
\[(n-qk)a(n-q+1,k))\]
permutations with $k$ A$q$C's.

If $n$ does not belong to an A$(q-1)$C,  then the same reasoning, but without the exception, gives that there are $1+(n-qk)$ positions in which to insert $n+1$ to get a permutation with $k$  A$q$Cs. In this case we get
\[(1+n-qk)(a(n,k)-a(n-q+1,k)\]
permutations with $k$
A$q$C's.

Adding these two numbers we get in this case 
\begin{equation}\label{eq:sum2}
(1+n-qk)(a(n,k)-a(n-q+1,k))\end{equation}
permutations with $k$ A$q$C's.

\smallskip\noindent 
{\it Case $\pi_{n+1,k}$'s from $\pi_{n,k-1}$'s}:  Now we have to create a new
A$q$C by inserting $n+1$ in a $\pi_{n,k-1}$, and this means that $(n-q+2,n-q+3,\ldots,n)$ must be a $(q-1)$-cycle in $\pi_{n,k-1}$ and we have to insert $n+1$ immediately after $n$. Thus in this case we get
\begin{equation}\label{eq:sum3}
a(n-k+1,k-1)
\end{equation}
permutations with $k$ A$q$C's. Adding (\ref{eq:sum1}), (\ref{eq:sum2}), and (\ref{eq:sum3}), we obtain the recurrence (\ref{eq:recurrence}).

\end{proof}

\section{Generating functions for A$q$C-free permutations}

The  (ordinary)  generating function $g(z)$ for derangements (A$1$C-free permutations) is known (see p.~199, Exercise 4 in \cite{comtet}) to satisfy the differential equation
\[z^2(1+z)g'(z)-(1-z^2)g(z)+1=0, \quad (g(0)=1).\]
Now let $q\ge 1$ and let
\[g(z)=\sum_{n=0}^{\infty} b_nz^n\]
be the {\it $($ordinary$)$ generating function} for the number $b_n$ of A$q$C-free permutations in $S_n$. 
Then $g(z)$ satisfies a first-order linear differential equation.

\begin{theorem}\label{th:genfctn}
For every $q\ge 1$, $g(z)$ satisfies the differential equation
\begin{equation}\label{eq:genfctn}
z^2(1+z^q)g'(z) -(1+z^q)(1-z-(q-1)z^q)g(z)+1- (q-1)z^q=0, \quad (g(0)=1).
\end{equation}
\end{theorem}

\begin{proof}
Start with the recurrence relation (\ref{eq:recurrence}) satisfied by the $b_n$,
multiply by $z^n$, and then sum over all $n\ge q$.
\end{proof}

Now let
\[G(z)=\sum_{n=1}^{\infty}b_n\frac{z^n}{n!}\]
be the {\it exponential generating function} for the numbers $b_n$ given in (\ref{eq:bformula}).

\begin{theorem}\label{th:expgenfctn}
For every $q\ge 1$, $G(z)$ satisfies the $q$th-order differential equation
\begin{equation}\label{eq:expgenfctn}
(1-z)G^{(q)}(z)-qG^{(q-1)}(z)-(q-1)G(z)=qw^{(q)}(z),\quad
(G^{(i)}(0)=i!\  (i=0,1,\ldots,q-1))\end{equation}
where 
\[w(z)=\sum_{i=1}^{\infty} \frac{(-1)^iz^{qi}}{(qi)!}.\]
\end{theorem}

\begin{proof}
We make use of (\ref{eq:recurrence}). 
We have
\begin{equation}\label{eq:egf1}
\frac{b_{qn}}{(qn)!}=\frac{b_{qn-1}}{(qn-1)!}+\frac{(q-1)}{qn(qn-1)\cdots (qn-q+1)}
\frac{b_{qn-q}}{(qn-q)!}+\frac{q(-1)^n}{(qn)!}.
\end{equation}

\smallskip\noindent
We also have for $j=1,2,\ldots,q-1$,
\begin{equation}\label{eq:egf2}
\frac{b_{qn+j}}{(qn+j)!}=\frac{b_{qn+j-1}}{(qn+j-1)!}+\frac{(q-1)}{(qn+j)(qn+j-1)\cdots (qn-q+j+1)}
\frac{b_{qn-q+j}}{(qn-q+j)!}
\end{equation}

\smallskip\noindent
Multiplying (\ref{eq:egf1}) by $z^{qn}$, and multiplying, for $j=1,2,\ldots,q-1$,  the $j$th equation of (\ref{eq:egf2}) by $z^{qn+j}$, and then summing for $n\ge 0$, we get
\begin{equation}\label{eq:egf3}
G(z)-\sum_{i=0}^{q-1}\frac{b_i}{i!}z^i=z\left(G(z)-\sum_{i=0}^{q-2}\frac{b_i}{i!}z^i\right)
+(q-1)z^qK(z)+q\sum_{i=1}^{\infty}\frac{(-1)^iz^{qi}}{(qi)!},
\end{equation}
 where
 \[K(z)=\sum_{i=0}^{\infty}\frac{b_i}{(i+1)(i+2)\cdots(i+q)}\frac{z^i}{{i!}}.\]
 Differentiating (\ref{eq:egf3}) $q$ times,  and using the fact that
 the $q$th derivative of $z^qK(z)$ equals $ G(z)$, we obtain (\ref{eq:expgenfctn}).
\end{proof}

The function $w(z)$ in Theorem \ref{th:expgenfctn}
can be expressed as a hypergeometric function, namely
\[w(z)=-\frac{z^q}{q!}F\left(1;\frac{q+1}{q},\frac{q+2}{q},\ldots,\frac{2q}{q};-\frac{z^q}{q^q}\right).\]
If $q=1$, we  obtain the known (see e.g. p.~106 in \cite{bona}) exponential generating function for the derangements, namely $G(z)=\frac{e^{-z}}{1-z}$.

\section{Adjacent cycles of more than one length}

Up to now we have considered permutations focusing on adjacent cycles of one given length $q$. There is a natural generalization of some of the results in the previous sections obtained by  replacing $q$ with  a finite set $Q$ of $m\ge 1$ lengths.
Thus let 
\[ Q=\{q_1,q_2,\ldots,q_m\} \mbox{ where } 1\le q_1<q_2<\cdots<q_m .\]
Let
\[a(n;k_1,k_2,\ldots,k_m)=\left|\{\pi\in S_n:\pi\mbox{ has $k_1$ A$q_1$C's, $k_2$ A$q_2$C's, $\ldots$, $k_m$ A$q_m$C's}\}\right|.\]
Note that 
$a(n;k_1,k_2,\ldots,k_m)=0$ if $q_1k_1+q_2k_2+\cdots+q_mk_m>n$.
As for the case of one length, we suppress the lengths $Q=\{q_1,q_2,\ldots,q_m\}$ in our notation.
The following theorem is the multi-length analogue of Theorem \ref{th:formula}.

\begin{theorem}\label{th:formulamulti}
We have
\begin{equation}\label{eq:formulamulti}
a(n;k_1,k_2,\ldots k_m)=\sum_{\{t_1,t_2,\ldots,t_m\ge 0:\sum_{j=1}^mq_jt_j\le n\}}
(-1)^{\sum_{j=1}^m(k_j+ t_j)}
\prod_{j=1}^m
{{t_j}\choose k_j} 
\frac{\left(n-\sum_{j=1}^m (q_j-1)t_j\right)!}{\prod_{j=1}^mt_j!}.
\end{equation}
\end{theorem}

\begin{proof} The proof is analogous to the proof of Theorem \ref{th:formula} and so we omit some of the details. Moreover, we illustrate the proof only in the case that $m=3$
from which it is clear how to proceed in general.
For $m=3$ we write (\ref{eq:formulamulti}) in the form
\begin{equation}\label{eq:formulamulti2}
\footnotesize
a(n;i,j,k)=\sum_{\{s,t,u\ge 0: ps+qt+ru\le n\}}
(-1)^{(i+j+k)+(s+t+u)}
{{s}\choose i}{{t}\choose j}{{u}\choose k}
\frac{\left(n-(p-1)s-(q-1)t-(r-1)u\right)!}{s!t!u!}.
\end{equation}

Corresponding to  the triangular system (\ref{eq:newtriangle}) we obtain in a similar way the  system of equations valid for all $s,t,u$ with $ps+qt+ru\le n$:

\begin{equation}\label{eq:trianglemulti}
\sum_{\{\alpha,\beta,\gamma\ge 0: p\alpha+q\beta+r\gamma\le n\}}
{{\alpha}\choose s}{{\beta}\choose t}{{\gamma}\choose u}a(n;\alpha,\beta,\gamma)
=\frac{\left(n-(p-1)s-(q-1)t-(r-1)u\right)!}{s!t!u!}.\end{equation}

To find the solution of this system, we require the following lemma.

\begin{lemma}\label{lem:trianglemulti}
Let $n,p,q,r$ be nonnegative integers with $p,q,r\le n$, and let
$\alpha,\beta,\gamma$ be nonnegative integers with $p\alpha+q\beta+r\gamma\le n$.
Then
\begin{equation}
\label{eq:multilem}
\sum_{\{s,t,u\ge 0:ps+qt+ru\le n\}}(-1)^{s+t+u}{s\choose i}{{\alpha}\choose s}{t\choose j}{{\beta}\choose t}{u\choose k}{{\gamma}\choose u}=
\left\{\begin{array}{ll}
(-1)^{i+j+k}&\mbox{ if $(\alpha,\beta,\gamma)=(i,j,k)$}\\
0&\mbox{ otherwise.}\end{array}\right.
\end{equation}\end{lemma}

\smallskip\noindent
{\it Proof of Lemma \ref{lem:trianglemulti}.}
The product
\[{s\choose i}{{\alpha}\choose s}{t\choose j}{{\beta}\choose t}{u\choose k}{{\gamma}\choose u}\]
is nonzero if and only if $i\le s\le \alpha$, $j\le t\le \beta$, and $k\le u\le \gamma$. All these triples belong to the summation in (\ref{eq:multilem}) because $ps+qt+ru\le p\alpha+q\beta+r\gamma\le n$. Thus the left hand side of (\ref{eq:multilem}) can be rewritten as a product of three sums:
\[
\left(\sum_{s=i}^{\alpha} (-1)^{s}{s\choose i}{{\alpha}\choose s}\right)
\left(\sum_{t=j}^{\beta} (-1)^t{t\choose j}{{\beta}\choose t}\right)
\left(\sum_{u=k}^{\gamma}(-1)^u{u\choose k}{{\gamma}\choose u}\right).\]
Considering the first sum in this product, we get
\[\sum_{s=i}^{\alpha} (-1)^{s}{s\choose i}{{\alpha}\choose s}
=(-1)^i{{\alpha}\choose i} \sum_{s=i}^{\alpha}(-1)^{s-i}{{\alpha-i}\choose {s-i}}=\left\{\begin{array}{ll}
0&\mbox{ if  $\alpha\ne i$}\\
(-1)^i&\mbox{ if $\alpha=i$,}\end{array}\right.\]
since the last summation leads to the full alternating sum of the binomial coefficients of the form ${{\alpha-i}\choose {\ast}}$. The other two sums in the product are similarly evaluated giving (\ref{eq:multilem}),  and thus proving the lemma.

Returning to the proof of the theorem, 
we multiply both sides of (\ref{eq:trianglemulti}) by 
\[(-1)^{i+j+k+s+t+u}{s\choose i}{t\choose j}{u\choose k}\]
 and then sum over all triples $(s,t,u)$ with $s,t,u\ge 0$ and $ps+qt+ru\le n$.  The result on the right hand side is the same as the right side of (\ref{eq:formulamulti2}), while on the left hand side we get
\[{\footnotesize
\sum_{\{s,t,u\ge 0: ps+qt+ru\le n\}}
(-1)^{(i+j+k)+(s+t+u)}
{{s}\choose i}{{t}\choose j}{{u}\choose k}
\sum_{\{\alpha,\beta\gamma\ge 0, p\alpha+q\beta+r\gamma\le n\}}
{{\alpha}\choose s}{{\beta}\choose t}{{\gamma}\choose u}a(\alpha,\beta,\gamma)},
\]
and this equals
%$(-1)^{i+j+k}$ times 
\[(-1)^{i+j+k}\sum_{\{\alpha,\beta,\gamma\ge 0, p\alpha+q\beta+r\gamma\le n\}}
a(\alpha,\beta,\gamma) \sum_{\{s,t,u\ge 0: ps+qt+ru\le n\}}
(-1)^{s+t+u}
{{s}\choose i}{{t}\choose j}{{u}\choose k}
{{\alpha}\choose s}{{\beta}\choose t}{{\gamma}\choose u},\]
or $a(i,j,k)$ in view of Lemma (\ref{lem:trianglemulti}).
\end{proof}

By taking $k_1=k_2=\cdots=k_m=0$ in (\ref{eq:formulamulti}), we obtain the number of permutations in $S_n$ having no adjacent cycles of any of the lengths $q_1,q_2,\ldots,q_m$, namely,
\[a(n;0,0,\ldots,0)=
\sum_{\{t_1,t_2,\ldots,t_m\ge 0:\sum_{j=1}^mq_jt_j\le n\}}
(-1)^{\sum_{j=1}^mt_j}
\frac{\left(n-\sum_{j=1}^m (q_j-1)t_j\right)!}{\prod_{j=1}^mt_j!}.\]
If we take $(q_1,q_2,\ldots,q_k)=(1,2,\ldots,k)$, then we obtain
\[a(n;0,0,\ldots,0)=
\sum_{\{t_1,t_2,\ldots,t_m\ge 0:\sum_{j=1}^m jt_j\le n\}}
(-1)^{\sum_{j=1}^mt_j}
\frac{\left(n-\sum_{j=1}^m (j-1)t_j\right)!}{\prod_{j=1}^mt_j!}.\]
These numbers can be viewed as {\it restricted derangement numbers} as the permutations they count are, in particular, derangements.

\section{Concluding remarks}

It is well known that, for a fixed $n$, the rencontre numbers $r(n,k)$
can be calculated using permanents. Specifically, let $J_n$ be the all 1s matrix of order $n$, and let $I_n$ be the identity matrix of order $n$. Then
\[\mbox{per}\left(xI_n+(J_n-I_n)\right)=\sum_{k=0}^n r(n,k)x^{n-k}\]
is the generating function for the numbers $r(n,0)=d_n, r(n,1),r(n,2),\ldots,r(n,n)=1$ (see \cite{riordan}, pp. 59, 184). 
By constructing appropriate matrices, we can also use permanents to 
obtain, for fixed $n$ and $q$, the generating polynomial for the numbers $a(n,0),a(n,1),a(n,2),\ldots,a(n,\left\lfloor \frac{n}{q}\right\rfloor)$. This procedure is more complicated and, as a result,
we do not formalize this but only illustrate the technique by a few examples.

\smallskip\noindent
{\bf Example.} Consider $n=6$ and $q=3$ and the matrix
\[\left[\begin{array}{cccccc}
1&x_1&1&1&1&1\\
1&1&x_1x_2&1&1\\
x_1&1&1&x_2x_3&1&1\\
1&x_2&1&1&x_3x_4&1\\
1&1&x_3&1&1&x_4\\
1&1&1&x_4&1&1\end{array}\right],\]
obtained by introducing an indeterminate for each A3C and placing it in the associated positions of the matrix. Calculating the permanent ($6!=720$ terms), we get a polynomial $p(x_1,x_2,x_3,x_4)$ in $x_1,x_2,x_3,x_4$.  In order to obtain the number of terms in this permanent which use all three $x_j$'s for each $j$, we first make the substitutions
$x_1^3=x$,  $x_2^3=x$,  $x_3^3=x$, $x_4^3=x$ in $p(x_1,x_2,x_3,x_4)$, and then make the substitutions $x_1=1$, $x_2=1$, $x_3=1$, and $x_4=1$.  This yields  the degree 3 polynomial 
\[697+22x+x^2.\]
Here $a(6,0)=697$, $a(6,1)=22$, and $a(6,2)=1$, That $a(6,2)=1$ 
follows since the only permutation  in $S_6$  with two  A3C's is $(123)(456)$. We verify that $a(6,1)=22$ as follows: 
There are 5 permutations whose unique A3C is (123), 6 permutations whose unique A3C is (234), 6 permutations whose unique A3C is(345), and 5 permutations whose unique A3C is (456). The remaining
$6!-1-22=697$ permutations have no A3C's.

In case we consider adjacent cycles of several lengths, that is, $Q=\{q_1,q_2,\ldots,q_m\}$, we can extend
the above procedure to obtain a multivariable generating polynomial for the numbers $a(n;k_1,k_2,\ldots,k_m)$. To illustrate this technique, let $n=5$ and let $Q=\{1,2,3,4,5\}$ (all possible lengths of adjacent cycles of  permutations in $S_5$). Consider the matrix 
\[\left[\begin{array}{ccccc}
x&y_1z_1u_1v_1&1&1&1\\
y_1&x&y_2z_1z_2u_1u_2v_1&1&1\\
z_1&y_2&x&y_3z_2z_3u_1u_2v_1&1\\
u_1&z_2&y_3&x&y_4z_3u_2v_1\\
v_1&u_2&z_3&y_4&x\end{array}\right].\]
The $x$'s, $y_j$'s, $z_j$'s, $u_j$'s, $v_1$ correspond to 
A1C's, A2C's, A3C's, A4C's, and A5C's, respectively. Let
$p(x,y_1,y_2,\ldots,v_1)$
be the multivariate polynomial obtained by calculating the permanent of this matrix.
Replacing each $y_j^2$ with $y$, each $z_j^3$ with $z$, each $u_j^4$ with $u$, and $v_1^5$ with $v$, and then in the resulting polynomial,  replacing each remaining subscripted variable with 1,
we obtain
\[f(x,y,z,u,v)=34+34x+6y+z+v+17x^2+6xy+2xu+2yz+6x^3+3xy^2+3x^2z+4x^3y+x^5,\]
the 5-variate generating polynomial for $S_5$ where $x,y,z,u,v$ mark the A1C's, A2C's, A3C's, A4C's, and A5C's, respectively. For example,
$3xy^2$ counts the permutations $(1)(23)(45)$, $(12)(3)(45)$,  and $(12)(34)(5)$, while $z$ counts the permutation $(15)(234)$.


\begin{thebibliography}{99}

\bibitem{bona} M.~Bona, {\it Combinatorics of Permutations}, Chapman \& Hall/CRC Press, Boca Raton, FL, 2004.

\bibitem{rab} R.A.~Brualdi, {\it Introductory Combinatorics} (5th edition), Pearson Prentice Hall, Upper Saddle River, NJ, 2010.


\bibitem{comtet} L.~Comtet, {\it Advanced Combinatorics},  D.~Reidel, Dordrecht, 1974.

\bibitem{riordan}J.~Riordan, {\it An Introduction to Combinatorial Analysis}, Princeton University Press, Princeton, NJ, 1978.


\end{thebibliography}
\end{document}